\documentclass[12pt]{article}

\usepackage{amssymb}
\usepackage{amsmath}
\usepackage{amsthm}
\usepackage{amsfonts}

\newcommand{\mmp}{\mathbb{P}}

\newcommand{\od}{\overset{d}{=}}
\newcommand{\dod}{\overset{d}{\to}}
\newcommand{\tp}{\overset{P}{\to}}

\newcommand{\me}{\mathbb{E}}
\newcommand{\mr}{\mathbb{R}}
\newcommand{\mn}{\mathbb{N}}

\newcommand{\lin}{\underset{n\to\infty}{\lim}}

\newcommand{\lix}{\underset{x\to\infty}{\lim}}

\newcommand{\lit}{\underset{t\to\infty}{\lim}}

\newtheorem{thm}{Theorem}[section]
\newtheorem{lemma}[thm]{Lemma}

\newtheorem{assertion}[thm]{Proposition}

\theoremstyle{definition}

\theoremstyle{remark}
\newtheorem{rem}[thm]{Remark}
\begin{document}
\title{Regenerative compositions in the case of slow variation:
 A renewal theory approach}
\date{\today}

\author{Alexander Gnedin\footnote{School of Mathematical Sciences, Queen Mary University of London,
 e-mail: a.gnedin@qmul.ac.uk}~~~and~~
Alexander Iksanov\footnote{ Faculty of Cybernetics, National T.\
Shevchenko University of Kiev, 01033 Kiev, Ukraine, e-mail:
iksan@univ.kiev.ua}}

\maketitle

\begin{abstract}
\noindent A regenerative composition structure is a
%coherent
sequence of ordered partitions derived from the range of a
subordinator by a natural sampling procedure.
%version of Kingman's paintbox correspondence.
In
this paper, we extend previous studies \cite{GBarbour,  GIM,
GnePitYor1} on  the asymptotics of the number of blocks $K_n$ in
the composition of integer $n$, in the case when the L{\'e}vy
measure of the subordinator has a property of slow variation at
$0$. Using tools from the renewal theory the limit laws for $K_n$
are obtained in terms of integrals involving the Brownian motion
or stable processes. In other words, the limit laws are either
normal or other stable distributions, depending on the
%$\alpha$-stable, with $\alpha\in (1,2)$,
behavior of the tail of L{\'e}vy measure at $\infty$.
 %Limit
%distributions for the number of singleton blocks are obtained in
%terms of integrals involving the Brownian motion or stable
%processes, respectively.
Similar results are also derived for the number of singleton
blocks.
\end{abstract}


\section{Introduction}

Let $S:=(S(t))_{t\geq 0}$ be a subordinator (i.e. an increasing L{\'e}vy process) with $S(0)=0$, zero drift, no killing and a nonzero L{\'e}vy measure $\nu$
on ${\mathbb R}_+$.
The closed range $\cal R$ of the process $S$ is a regenerative subset of ${\mathbb R}_+$ of zero Lebesgue measure.
The range $\cal R$
splits the positive halfline in infinitely many disjoint component intervals that form an open set
$(0,\infty)\setminus {\cal R}$. These component intervals, further called gaps, are associated with jumps of $S$.
Let $E_1,\ldots, E_n$ be a sample drawn  independently of $S$ from the standard exponential distribution.
Each sample point $E_j$ falls in
the generic gap $(a,b)$
with probability $e^{-a}-e^{-b}$.
A gap is said to be occupied if it contains at least one of $n$ sample points.
The sequence of positive occupancy numbers of the gaps, recorded in the natural order of the gaps,
 is a composition (ordered partition) ${\cal C}_n$ of integer $n$.
The number $K_n$ counting the blocks of the composition is equal to the number of gaps occupied
by at least one sample point,
and
the number $K_{n,r}$ counting the blocks of size $r$ is the number of gaps occupied by exactly $r$ out of $n$ sample points,
so that $K_n=\sum_{r=1}^n K_{n,r}$ and $n=\sum_{r=1}^n r K_{n,r}$.

The sequence of random compositions  $({\cal C}_n)_{n\in{\mathbb N}}$ derived in this way has the following two recursive properties.
The first property of {\it sampling consistency} is a form of exchangeability:
removing a randomly chosen sample point from the first $n$ points maps ${\cal C}_n$ in a distributional copy of ${\cal C}_{n-1}$.
This property is obvious from the construction and exchangeability, because removing a random point has the same effect as restricting to $n-1$ points $E_1,\dots,E_{n-1}$.
The second property is specific for regenerative ${\cal R}$.
Consider  composition ${\cal C}_n$ of size $n$ and suppose it occurs that the first part,  which is the number of sample points in the leftmost occupied gap,
is some $m<n$, then deleting this part yields a composition on  $n-m$ remaining points which is a distributional copy of  ${\cal C}_{n-m}$.
This property is a combinatorial counterpart of the regenerative property of $\cal R$, therefore
sequences $({\cal C}_n)_{n\in{\mathbb N}}$ are called
{\it regenerative composition structures} \cite{RCS}.
In particular, for suitable choice of $\nu$ the construction generates an ordered version of the familiar Ewens-Pitman two-parameter partition structure \cite{RCS}.

The regenerative composition structures  appear in a variety of contexts related to partition-valued processes and random discrete distributions
(see \cite{RegenCS} for a survey).
To bring the construction of compositions in a more conventional context we may consider a random distribution function
$F(t)=1-\exp(-S(t))$  on positive reals, also known as a {\it neutral to the right prior} \cite{Walker}.
Since $F$ has atoms, an independent $n$-sample from the distribution (defined conditionally given $F$) will have clusters of repeated values, thus we may
define a composition  ${\cal C}_n$ by recording the multiplicities in the order of increase of the values represented in the sample.

Limit distributions for $K_n$ (properly centered and normalized) were studied under various assumptions on $S$.
When the L{\'e}vy measure $\nu$ is finite, the process $S$ is  compound Poisson, and $\cal R$ is the discrete set of atoms of a renewal point process.
 A characteristic
feature of this case is that almost all of the gaps within $[0, S(\log n)]$ are occupied, and these give a dominating contribution to $K_n$.
In the compound Poisson case there is a rather complete theory \cite{Gne, GIM, GINR}  surveyed in \cite{GIM2}.

In the case of infinite L{\'e}vy measure the asymptotic behaviour of $K_n$ is related to that  of the tail
$\nu[x,\infty)$ as $x\to 0$;
concrete results for infinite $\nu$ have been obtained under the assumption of regular variation.
If $\nu[x,\infty)$ varies regularly at $0$  with positive index, both $K_n$ and $K_{n,1}$ may be
normalized by the same constant (no centering required) to entail  convergence to multiples of
the same random variable,
which may be represented as  the exponential functional of a subordinator  \cite{GnePitYor2}.
This case is relatively easy, because the number of occupied gaps within the partial range in $[S(t_1),S(t_2)]$ is of the same order of growth, as $n\to\infty$, for every time interval $0\leq t_1<t_2\leq\infty$.

The case of infinite L{\'e}vy measure with $\nu[x,\infty)$ slowly varying at $0$ is much more delicate,
because the occupied gaps do not occur that uniformly as in the case of regular variation with positive index,
nor the primitive gap-counting works: unlike the compound Poisson case, $\cal R$  has topology of a Cantor set.
Each ${\mathbb E}K_{n,r}$
is then of the order of growth smaller than that of ${\mathbb E}K_n$,
and the convergence of $K_n$ and $K_{n,r}$'s  requires nontrivial centering.
Normal limits were shown in \cite{GnePitYor1}
in the special case of subordinators which, like the gamma subordinators, have
$\nu[x,\infty)$  of logarithmic growth
as $x\to 0$.
Normal limits for $K_n$ for wider families of slowly varying functions were obtained in \cite{GBarbour} under
the assumption that the subordinator
has finite variance and the Laplace exponent of $S$ satisfies certain smoothness and growth conditions.
It turned that the case of slow variation required a further division, with qualitatively different scaling functions in each subcase \cite{GBarbour}.

The method of   \cite{GBarbour} relied on linearization of
the compensator process
for the number of occupied gaps contained in $[0,S(t)]$,
and application of the functional central limit theorem for $S$.
In this paper we develop a different approach to the asymptotics
of $K_n$ in the case of $\nu[x,\infty)$ slowly varying at $0$.
As in \cite{GBarbour}, we analyse the compensator process, but instead of linearizing it, apply
the renewal theory and functional limit theorems for the
{\it first passage times} process, that is random function inverse to $S$.
This gives a big technical advantage, enabling us
to simplify arguments and  to increase generality.
The class of slowly
varying functions covered in this paper will be larger than that
in \cite{GBarbour, GnePitYor1}.
In particular, we will
omit the  assumption of finite variance of $S$ and find conditions
on the L{\'e}vy measure $\nu$ to guarantee a weak convergence of $K_n$
to the normal or some other stable distributions.
A similar approach, with a discrete-time version of the compensator will be applied also
in the case of finite L{\'e}vy
measure, leading to
 known asymptotics \cite{Gne, GIM, GINR} in a
more compact way. We shall also identify the limit distribution
for $K_{n,1}$ in terms of an integral involving a random process
corresponding to the limit law of $K_n$. With some additional
effort, our approach to the limit laws of $K_{n,1}$ could be
extended to $K_{n,r}$ for all $r\geq 1$,
but to avoid technical complications
we do not pursue this extension here, as our main focus is the
development of the new  method.

\section{Preliminaries}

As in much of the previous work, it will be convenient to
poissonize the occupancy model, that is to replace the exponential
sample $E_1,\dots,E_n$ of fixed size $n$ by atoms of an
inhomogeneous Poisson  process $(\pi_t(x))_{x\geq 0}$, which is
independent of $S$ and has the intensity measure $\lambda_t({\rm
d}x)=te^{-x}{\rm d}x$ on ${\mathbb R}_+$. The total number of
atoms, $\pi_t:=\pi_t(\infty)$, has then the Poisson distribution
with mean $t$. We will use the notation $K(t):=K_{\pi_t}$ for the
number of gaps occupied by at least one atom of the Poisson
process, and $K(t,r)$  for the number of gaps occupied by exactly
$r$ such atoms.

Introduce
$$\Phi(t):=\int_{[0,\,\infty)}\big(1-\exp\{-t(1-e^{-x})\}\big)\nu({\rm  d}x), \ \ t>0.$$
By Proposition 2.1 of \cite{GBarbour}  the increasing process
$(A(t, u))_{u\in[0,\infty]}$ defined by
$$A(t,u):=\int_{[0,\,u]}\Phi\big(te^{-S(v)}\big){\rm d}v$$
is the compensator of the increasing process which counts the
number of gaps in  $[0,S(u)]\setminus {\cal R}$, that are occupied
by at least one atom of the Poisson sample. Similarly, one can
check that
$$A^{(r)}(t, u):=\int_{[0,\,u]} \Phi^{(r)}\big(te^{-S(v)}\big){(te^{-S(v)})^r\over r!} {\rm d}v,$$
where $\Phi^{(r)}$ denotes the $r$th derivative of $\Phi$, is the
compensator of the increasing process which counts the number of
gaps in $[0,S(u)]\setminus {\cal R}$ that contain exactly $r$
Poisson atoms.

The asymptotics of $K(t)$ and $K(t,1)$ for large $t$ is closely
related to the terminal values of the compensators
\begin{equation}\label{co}
A(t):=A(t,\infty)=\int_{[0,\,\infty)}\Phi(te^{-S(v)}){\rm
d}v=\int_{[0,\,\infty)}\Phi(te^{-s})\,{\rm d}T(s)
\end{equation}
and
\begin{equation}\label{co1} A^{(1)}(t):=A^{(1)}(t,\infty)=\int_{[0,\infty)}\Phi^\prime(te^{-S(v)})te^{-S(v)}\,{\rm d}v,
\end{equation}
where $$T(s):=\inf\{t\geq 0: S(t)>s\}, \ \ s\geq 0$$
is the passage  time of $S$ through level $s$.

Like in many other models of allocating `balls' in `boxes' with random probabilities of `boxes',
the variability of $K_n$
has two sources: the randomness of $\cal R$, and the randomness involved in drawing a sample
conditionally given $\cal R$.
For regenerative compositions it has been shown, in various forms, that the first factor of variability
has a dominant role.
See, for instance, \cite{GIM} for the compound Poisson case.
We shall confirm the phenomenon in the case of slow variation
by showing that  $A(n)$ absorbs a dominant part of the variability, to the extent that
$A(n)$
and $K_n$, normalized and centered by the same
constants, have the same limiting distributions.

Throughout we shall assume  that the
function $\Phi$ satisfies one of the following three
conditions:
\newline {\sc Condition A}:
\begin{equation*}
\varphi(t):=\Phi(e^t) \ \sim \ t^\beta L_1(t), \ \ t \to \infty,
\end{equation*}
for some $\beta\in [0,\infty)$ and some function $L_1$ slowly varying at
$\infty$.  For $\beta=0$ we assume $\lit L_1(t)=\infty$.
\newline {\sc
Condition B}: $\varphi(t)$ belongs to de Haan's class $\Gamma$,
i.e., there exists a measurable function $h:\mr\to (0,\infty)$
called the {\it auxiliary function} of $\varphi$ such that
$$\underset{t\to\infty}{\lim}{\varphi(t-uh(t))\over \varphi(t)}=e^{-u} \ \ \text{for all} \ u\in\mr
{\rm ~~~and~~~} \lit h(t)=\infty.
$$
\newline {\sc
Condition C}: $\varphi(t)$ is a bounded function,
which holds if and only if  the L{\'e}vy measure $\nu$ is finite, {\it i.e.}
 $S$ is a compound Poisson process.

Let
$$\widehat{\Phi}(t):= \int_0^\infty (1-e^{-tx})\nu({\rm d}x)$$
denote the conventional Laplace exponent of $S$. Then
$$\Phi(t)\sim \widehat{\Phi}(t), \ \ t\to\infty.$$ Therefore, conditions
A, B and C can be equivalently formulated with $\widehat{\Phi}$ in
place of $\Phi$. Either of the conditions implies that
%\begin{equation}\label{222101}
%\Phi(y) \ \sim \ \ell(y), \ \ y\to\infty,
%\end{equation}
%for some $\ell$ slowly varying at $\infty$
$\Phi$ is a function slowly varying at $\infty$, hence
%(see Lemma \ref{5050} for this assertion under Condition B). Hence, by the combination
%of Theorem $1.7.1'$ (applied to $t\mapsto \Phi(t)/t$) and Theorem
%1.7.2b in \cite{BGT} (applied to $x\mapsto
%\int_0^x\nu[y,\infty]{\rm d}y$),
by Karamata's Tauberian theorem
\begin{equation}\label{2} \nu[x,\infty] \ \sim \ \Phi(1/x), \ \ x\to 0.
\end{equation}

%Further asymptotic properties of the L\'{e}vy measure $\nu_0$
%under Conditions A-C can be found in Lemma ...

It should be noted that there are slowly varying functions which
satisfy neither of the conditions A, B and C, for instance
functions which behave like
$$\Phi(t)\sim
\exp\left(\int_2^t {|\sin u|\over \log u}{\rm d}u\right), ~~~t\to\infty.$$
However, functions that are not covered by one of the
conditions A, B and C are rather exceptional.
The case $\beta=0$ of Condition A covers the functions called in \cite{GBarbour}
`slowly growing', and the case $0<\beta <\infty$ `moderately growing'.

The rest of the paper is organized as follows. In Section
\ref{gen} we consider the case of infinite $\nu$ when one of the
conditions A or B holds, and derive the limit
distributions for $K_n$ and $K_{n,1}$. In Section \ref{com} we give
a simplified treatment of  the much studied case \cite{Gne,
GINR, GIM} when $\nu$ is a finite measure, that is Condition C
holds. Finally,  some auxiliary facts are collected in the Appendix.

\section{Subordinators with infinite L{\'e}vy measure}\label{gen}

\subsection{Convergence of $K_n$}

Our first main result concerns `slowly growing' or `moderately
growing' functions $\Phi$ of slow variation, behaving like e.g.
$\Phi(t)\sim
(\log_k
t)^\beta(\log_m t)^\delta$, for some $\beta>0$, $\delta\geq 0$,
$k\in\mn$, $m\in\mn_0$, where $\log_i x$ denotes the $i$-fold
iteration of the natural logarithm.

Introduce the moments of $S(1)$
$${\tt s}^2:={\rm Var}\,S(1)=\int_{[0,\,\infty)}x^2\nu({\rm d}x), ~~~{\tt m}:={\mathbb E}S(1)=
\int_{[0,\,\infty)} x \nu({\rm d}x).$$ Note that ${\tt m}<\infty$
under the assumptions of all the subsequent theorems.
\begin{thm}\label{main}
Suppose Condition A holds.
\newline {\rm (a)}
Suppose ${\tt s}^2<\infty$. If $\beta>0$ then $${K_n-{\tt
m}^{-1}\int_{[1,\,n]}y^{-1}\Phi(y){\rm d}y\over \sqrt{{\tt
s}^2{\tt m}^{-3}\log n}\Phi(n)} \ \dod \ \beta
\int_{[0,1]}Z(1-y)y^{\beta-1}{\rm d}y, \ \ n\to\infty,$$ where
$\big(Z(y)\big)_{y\in[0,1]}$ is the Brownian motion, and if
$\beta=0$ then the limiting random variable is $Z(1)$. %then the
%limiting distribution of ${K_n-b_n\over a_n}$ is standard normal,
%for the choice of constants as
%\begin{equation}\label{102}
%b_n={1\over {\tt m}}\int_1^n{\Phi(z)\over z}{\rm d}z
%\end{equation}
%%where ${\tt m}:=\me S(1)$,
%and
%$$a_n=\Phi(n)\sqrt{{{\tt s}^2\over (2\beta+1){\tt m}^3}\log n}\ \sim \ \sqrt{{{\tt s}^2\over {\tt m}^3}\int_1^n {\Phi^2(z)\over z}{\rm
%d}z}.$$
%\sim \ \sqrt{{{\tt s}^2\over {\tt m}^3}\int_1^n {\Phi^2(z)\over z}{\rm d}z},$$
\newline {\rm (b)} Suppose ${\tt s}^2=\infty$ and
$$\int_0^x y^2 \nu({\rm d}y) \ \sim \ L(x), \ \ x\to\infty,$$ for some $L$ slowly varying at
$\infty$. Let $c(x)$ be any positive function satisfying
$\lix\,xL(c(x))/c^2(x)=1$. If $\beta>0$ then
$${K_n-{\tt m}^{-1}\int_{[1,\,n]}y^{-1}\Phi(y){\rm d}y\over
{\tt m}^{-3/2}c(\log n)\Phi(n)} \ \dod \ \beta
\int_{[0,1]}Z(1-y)y^{\beta-1}{\rm d}y, \ \ n\to\infty,$$ where
$(Z(y))_{y\in[0,1]}$ is the Brownian motion, and if $\beta=0$ then
the limiting random variable is $Z(1)$.
%then, with $b_n$ as in \eqref{102} and
%$$a_n=((2\beta+1){\tt m}^3)^{-1/2}\Phi(n)c(\log n)\ \sim \ \sqrt{{\tt m}^{-3}L(c(\log n))\int_1^n
%{\Phi^2(z)\over z}{\rm d}z},$$ where  the limiting distribution of
%${K_n-b_n\over a_n}$ is standard normal.
\newline {\rm (c)} Suppose
\begin{equation}\label{domain1}
\nu[x,\infty] \ \sim \ x^{-\alpha}L(x), \ \ x\to \infty,
\end{equation}
for some $L$ slowly varying at $\infty$ and $\alpha\in (1,2)$. Let
$c(x)$ be any positive function satisfying $\lix
\,xL(c(x))/c^\alpha(x)=1$. If $\beta>0$ then
$${K_n-{\tt m}^{-1}\int_{[1,\,n]}y^{-1}\Phi(y){\rm d}y\over
{\tt m}^{-(\alpha+1)/\alpha}c(\log n)\Phi(n)} \ \dod \ \beta
\int_{[0,1]}Z(1-y)y^{\beta-1}{\rm d}y, \ \ n\to\infty,$$ where
$\big(Z(y)\big)_{y\in [0,1]}$ is the $\alpha$-stable L\'{e}vy
process such that $Z(1)$ has characteristic function
\begin{equation}\label{st1}
u\mapsto \exp\{-|u|^\alpha
\Gamma(1-\alpha)(\cos(\pi\alpha/2)+i\sin(\pi\alpha/2)\, {\rm
sgn}(u))\}, \ u\in\mr,
\end{equation}
and if $\beta=0$ then the limiting random variable is $Z(1)$.
\end{thm}
\vskip0.2cm
\begin{rem}\label{1002}
%For $\alpha=2$
%formula \eqref{1066} also gives the normalizing
%constants $a_n$
%for parts (a) and (b) of the theorem,
%since
% ${\tt s}^2<\infty$ implies $\int_0^x y^2\nu({\rm d}y) \ \sim \
%L(x) \ \to \ {\tt s}^2$, as $x\to\infty$.

Set $J:=\beta \int_{[0,1]}Z(1-y)y^{\beta-1}{\rm d}y$. By Lemma
\ref{ri},
$$\log \me \exp({\rm i}tJ)= \int_{[0,\,1]} \log \me\exp({\rm i}t(1-x)^\beta Z(1)){\rm
d}x.$$ Hence $J \od (\alpha \beta+1)^{-1/\alpha}Z(1)$ where the
case $\alpha=2$ corresponds to parts (a) and (b) of Theorem
\ref{main}.

In the definition of constants  $\Phi$ can be replaced
by the Laplace exponent $\widehat{\Phi}$, since the difference
between the functions vanishes at $\infty$
(see \cite{GnePitYor2} Lemma A.1 or \cite{GBarbour} Lemma 2.3).
\end{rem}
\vskip0.2cm
\begin{proof}
Under the assumptions of part (a) we denote by $Z(\cdot)$ the
Brownian motion and set $g(t):=\sqrt{{\tt s}^2{\tt m}^{-3} t}$, under
the assumptions of part (b) we denote by $Z(\cdot)$ the
Brownian motion and set $g(t):={\tt m}^{-3/2}c(t)$ and under the
assumptions of part (c) we denote by $Z(\cdot)$ the
$\alpha$-stable L\'{e}vy process such that $Z(1)$ has
characteristic function \eqref{st1}, and set $g(t):={\tt
m}^{-1-1/\alpha}c(t)$.

For later use we note that $g$ varies regularly at $\infty$ with index
$1/\alpha$, where $\alpha=2$ corresponds to the cases (a) and (b).
This follows from Theorem 1.5.12 in \cite{BGT} which is a result
on asymptotic inverses of regularly varying functions. Further we
note that in the cases (b) and (c) $g(x)$ grows faster than
$\sqrt{x}$. In the latter case, this follows trivially from the
regular variation of $g$ with index $1/\alpha$, $\alpha\in (1,2)$.
In the former case, we have $\int_{[0,\,x]}y^2\nu({\rm d}y) \sim
L(x)$, where $\lix L(x)=\infty$, and $c(x)$ satisfies $\lix
{xL(c(x))\over c^2(x)} =1$. Since $\lix L(c(x))=\infty$ we infer
$\lix {x\over c^2(x)}=0$.

\noindent {\sc Step 1}: We first investigate convergence in
distribution of, properly normalized and centered, $A(t)$, as
$t\to\infty$. Recalling the notation $\varphi(t)=\Phi(e^t)$,
representation \eqref{co}
can be rewritten using integration by parts for the Lebesgue-Stieltjes integral as follows
$$A(e^t)-\varphi(0)T(t)=\int_{[0,\,t]}T(t-z){\rm d}\varphi(z)+\int_{[t,\,\infty)}\varphi(t-z){\rm
d}T(z)=:A_1(t)+A_2(t).$$ Now we want to look at the asymptotic
behavior of $A_2(t)$, as $t\to\infty$. Since
$\Phi^\prime(0)<\infty$ (this is equivalent to the characteristic property
$\int_{[0,\infty)}\min(y,1)\nu({\rm d}y)<\infty$ which holds for
every L\'{e}vy measure $\nu$), the function $\varphi(t)$ is
integrable on $(-\infty,0]$, which together with its monotonicity
ensures that it is directly Riemann integrable on $(-\infty, 0]$.
Therefore, by the key renewal theorem
\begin{equation}\label{inter}
\me A_2(t)= \me \int_{[t,\,\infty)}\varphi(t-z){\rm d}T(z) \ \to \
{\tt m}^{-1}\int_{(-\infty,\, 0]}\varphi(z){\rm d}z<\infty, \ \
t\to\infty.
\end{equation}
{\sc Case $\beta>0$}. It is known (see Theorem 2a in
\cite{Bing72}) that
\begin{equation}\label{16}
W_t(\cdot):={T(t\cdot)-{\tt m}^{-1}(t\cdot)\over g(t)} \Rightarrow
\ Z(\cdot), \ \ t\to\infty,
\end{equation}
in $D[0,\infty)$ in the Skorohod $M_1$-topology. In particular,
\begin{equation}\label{234}
{T(t)-{\tt m}^{-1}t \over g(t)\varphi(t)} \ \tp \ 0, \ \
t\to\infty.
\end{equation} %The limit relation $$\lit
%{\varphi(tu)\over \varphi(t)}=u^\beta$$ holds locally uniformly on
%$[0,\infty)$. Hence
%$$u_t(\cdot):={\varphi(t\cdot)\over \varphi(t)} \
%\Rightarrow \ (\cdot)^\beta=:u(\cdot), \ \ t\to\infty$$ in
%$D(0,\infty)$.
To apply Lemma \ref{impo} take $X_t=W_t$ and let $Y_t$ and $Y$ be
random variables with distribution functions $\mmp\{Y_t\leq
y\}={\varphi(ty)\over \varphi(t)}=:u_t(y)$ and $\mmp\{Y\leq
y\}=y^\beta=:u(y)$, $0\leq y\leq 1$. Then, as $t\to\infty$,
\begin{eqnarray*}
{A_1(t)-{\tt m}^{-1}\int_{[0,\,t]}(t-z){\rm d}\varphi(z)\over
g(t)\varphi(t)}&=& \int_{[0,1]}W_t(1-y){\rm d}u_t(y)\\&\dod&
\int_{[0,1]}Z(1-y){\rm
d}u(y)\\&=&\beta\int_{[0,1]}Z(1-y)y^{\beta-1}{\rm d}y=J.
\end{eqnarray*}
Recalling \eqref{inter} and \eqref{234} we obtain
\begin{equation}\label{103}
{A(e^t)-{\tt m}^{-1}\bigg(\int_{[0,\,t]}(t-z){\rm
d}\varphi(z)+\varphi(0)t\bigg)\over g(t) \varphi(t)} \ \dod \ J.
\end{equation}
Noting that
$$\int_{[1,\,e^t]}y^{-1}\Phi(y){\rm
d}y=\int_{[0,\,t]}\varphi(y){\rm d}y=\int_{[0,\,t]}(t-z){\rm
d}\varphi(z)+\varphi(0)t$$ and replacing in \eqref{103} $e^t$ by
$t$ concludes the proof of Step 1 in the  case $\beta>0$.\newline
{\sc Case $\beta=0$}. We have, for $\varepsilon\in (0,1)$
\begin{eqnarray}\label{1115}
{A_1(t)-{\tt m}^{-1}\int_{[0,\,t]}(t-z){\rm d}\varphi(z)\over g(t)
\varphi(t)}&=& {\int_{[0,\,\varepsilon]}W_t(1-y){\rm
d}\varphi(ty)\over \varphi(t)}\\&+& {\int_{[\varepsilon,\,
1]}W_t(1-y){\rm d}\varphi(ty)\over \varphi(t)
}\nonumber\\&=:&J_1(t,\varepsilon)+J_2(t,\varepsilon)\nonumber.
\end{eqnarray}
We first show that
\begin{equation}\label{1000}
\underset{\varepsilon\downarrow
0}{\lim}\,\lit\,J_1(t,\varepsilon)=Z(1) \ \ \text{in
distribution}.
\end{equation}
To this end, we use the bounds
$$\underset{y\in
[0,\varepsilon]}{\inf}\,W_t(1-y){\varphi(\varepsilon
t)-\varphi(0)\over \varphi(t)}\leq J_1(t,\varepsilon)\leq
\underset{y\in [0,\varepsilon]}{\sup}\,W_t(1-y)
{\varphi(\varepsilon t)\over \varphi(t)}.$$ Recall that the
function $h_1: D[0,\infty)\to \mr$ defined by
$h_1(x):=\underset{y\in [0,\varepsilon]}{\sup}\,x(y)$ is
$M_1$-continuous (see Section 13.4 in \cite{Whitt2}). Hence, in
view of \eqref{16} we conclude that, as $t\to\infty$, the
right-hand side converges in distribution to $\underset{y\in
[0,\varepsilon]}{\sup}\,Z(1-y)$. This further converges to $Z(1)$
on letting $\varepsilon \downarrow 0$. A similar argument applies
to the left-hand side, and \eqref{1000} has been proved.

Using the inequality
$$\underset{y\in
[\varepsilon,1]}{\inf}\,W_t(1-y){\varphi(t)-\varphi(\varepsilon
t)\over \varphi(t)}\leq J_2(t,\varepsilon)\leq \underset{y\in
[\varepsilon, 1]}{\sup}\,W_t(1-y) {\varphi(t)-\varphi(\varepsilon
t)\over \varphi(t)}$$ and arguing in much the same way as above we
conclude that $\lit\,J_2(t,\varepsilon)=0$ in distribution. This
together with \eqref{inter} allows us to conclude that
\begin{equation}\label{10003}
{A(e^t)-{\tt m}^{-1}\bigg(\int_{[0,\,t]}(t-z){\rm
d}\varphi(z)+\varphi(0)t\bigg)\over g(t) \varphi(t)} \ \dod \
Z(1).
\end{equation}
Replacing in this relation $e^t$ by $t$ completes the proof of
Step 1 in the case $\beta=0$.

\noindent {\sc Step 2}: Now we argue that the same convergence in
distribution holds with $A(t)$ replaced by $K(t)$. In other words,
we will prove that $${A(t)-K(t)\over g(\log t)\Phi(t)} \tp \ 0, \
\ t\to\infty.$$ Since in the cases (b) and (c) $g(x)$ grows faster
than $\sqrt{x}$ (see the beginning of the proof) it suffices to
show that
\begin{equation}\label{xxx}
{A(t)-K(t)\over \sqrt{\log t}\Phi(t)} \tp \ 0, \ \ t\to\infty.
\end{equation}
By Lemma 2.6 in \cite{GBarbour},
\begin{equation}\label{5555}
\me \big(A(t)-K(t)\big)^2 \ \sim \ {\tt m}^{-1}\int_{[1,\,t]}
u^{-1}\Phi(u){\rm d}u, \ \ t\to\infty.
\end{equation}
Hence $$\me \bigg({A(t)-K(t)\over \sqrt{\log t}\Phi(t)}\bigg)^2 \
\sim \ {\int_{[1,\,t]} u^{-1}\Phi(u){\rm d}u\over {\tt m}\log
t\Phi^2(t)}\leq {\log t\Phi(t)\over {\tt m}\log
t\Phi^2(t)}={1\over {\tt m}\Phi(t)},$$ and \eqref{xxx} follows by
Chebyshev's inequality.

\noindent {\sc Step 3}: The last step is `depoissonization', i.e.
passing from the Poisson process to the original fixed-$n$
exponential sample. Since $K(t)$ is nondecreasing, this is easy,
and the proof is omitted (see the proof of Theorem \ref{main4}
where the depoissonization is implemented for a non-monotone
function).
\end{proof}

Our second main result concerns `fast' functions of slow variation $\Phi$, which grow
faster than any power of $\log t$,
for instance
$\Phi(t)\sim   \exp(\gamma \log^\delta t)$ for some $\gamma>0$ and $\delta\in
(0,1)$.
\begin{thm}\label{main3}
Suppose Condition B holds.

\noindent {\rm (a)} \noindent Under the assumption of part (a) of
Theorem \ref{main}
$${K_n-{\tt m}^{-1}\int_{[1,\,n]}y^{-1}\Phi(y){\rm d}y \over \sqrt{{\tt
s}^2{\tt m}^{-3}\log n}\Phi(n)} \ \dod \
\int_{[0,\,\infty)}Z(y)e^{-y}{\rm d}y, \ \ n\to\infty,$$ where
$(Z(y))_{y\geq 0}$ is the Brownian motion.

 %If ${\tt s}^2
%%={\rm Var}\,S_1=\int_0^\infty x^2\nu({\rm d}x)
%<\infty$ then, with $b_n$ given in \eqref{102} and
%$$a_n=\Phi(n)\sqrt{{{\tt s}^2\over 2{\tt m}^3}h(\log n}),$$ the
%limiting distribution of ${K_n-b_n\over a_n}$ is standard normal.
\noindent {\rm (b)} Under the assumptions of part (b) of Theorem
\ref{main} $${K_n-{\tt m}^{-1}\int_{[1,\,n]}y^{-1}\Phi(y){\rm d}y
\over {\tt m}^{-3/2}c(\log n)\Phi(n)} \ \dod \
\int_{[0,\,\infty)}Z(y)e^{-y}{\rm d}y, \ \ n\to\infty,$$ where
$\big(Z(y)\big)_{y\geq 0}$ is the Brownian motion.

% If ${\tt s}^2=\infty$ and $$\int_0^x y^2 \nu({\rm d}y)
%\ \sim \ L(x), \ \ x\to\infty,$$ for some $L$ slowly varying at
%$\infty$, then, with $b_n$ given in \eqref{102} and
%$$a_n=(2{\tt m}^3)^{-1/2}\Phi(n)c(h(\log n)),$$ where $c(x)$ is any positive
%function satisfying $\lix\,xL(c(x))/c^2(x)=1$, the limiting
%distribution of ${K_n-b_n\over a_n}$ is standard normal.\newline
\noindent {\rm (c)} Under the assumptions of part (c) of Theorem
\ref{main} $${K_n-{\tt m}^{-1}\int_{[1,\,n]}y^{-1}\Phi(y){\rm d}y
\over {\tt m}^{-(\alpha+1)/\alpha}c(\log n)\Phi(n)} \ \dod \
\int_{[0,\,\infty)}Z(y)e^{-y}{\rm d}y, \ \ n\to\infty,$$ where
$\big(Z(y)\big)_{y\geq 0}$ is the $\alpha$-stable L\'{e}vy process
such that $Z(1)$ has characteristic function \eqref{st1}.

%If
%\begin{equation*}
%\nu[x,\infty] \ \sim \ x^{-\alpha}L(x), \ \ x\to \infty,
%\end{equation*}
%for some $L$ slowly varying at $\infty$ and $\alpha\in (1,2)$
%then, with $b_n$ as in \eqref{102} and
%\begin{equation*}
%a_n=(\alpha{\tt m}^{\alpha+1})^{-1/\alpha}\Phi(n)c(h(\log n)),
%\end{equation*}
%where $c(x)$ is any positive function satisfying $\lix
%\,xL(c(x))/c^\alpha(x)=1$, the limiting distribution of
%$(K_n-b_n)/a_n$ is the $\alpha$-stable law with characteristic
%function \eqref{st1}.
\end{thm}
\begin{rem}\label{1001}
Set $K:=\int_{[0,\,\infty)}Z(u)e^{-u}{\rm d}u$. By Lemma \ref{ri},
$$\log \me \exp({\rm i}tK)= \int_{[0,\infty)}\log \me\exp({\rm
i}te^{-x}Z(1)){\rm d}x.$$ Hence $K \od \alpha^{-1/\alpha}Z(1)$
where the case $\alpha=2$ corresponds to parts (a) and (b) of
Theorem \ref{main3}.
\end{rem}
\begin{proof}
We use the same notation as in the proof of Theorem \ref{main}. We
will only show that
\begin{equation}\label{1033}
{A(e^t)-{\tt m}^{-1}\bigg(\int_{[0,\,t]}(t-z){\rm
d}\varphi(z)+\varphi(0)t\bigg)\over g(h(t)) \varphi(t)} \ \dod \
K=\int_{[0,\,\infty)}Z(u)e^{-u}{\rm d}u,
\end{equation}
the rest of the proof being the same as in Theorem \ref{main}.
For any fixed $a>0$, we have
\begin{eqnarray}\label{105}
{A_1(t)-{\tt m}^{-1}\int_{[0,\,t]}(t-z){\rm d}\varphi(z)\over
g(h(t))\varphi(t)}&=& -\int_{[0,\,a]}W_{h(t)}(y){\rm d}v_t(y)\\&-&
\int_{[a,\,t/h(t)]}W_{h(t)}(y){\rm
d}v_t(y)\nonumber\\&=:&J_3(t,a)+J_4(t,a)\nonumber,
\end{eqnarray}
where $v_t(u):={\varphi(t-yh(t))\over \varphi(t)}$. To apply Lemma
\ref{impo} we take $X_t=W_{h(t)}$ and let $Y_t$ and $Y$ be random
variables with $\mmp\{Y_t>u\}=v_t(u)$ and $\mmp\{Y>u\}=e^{-u}$.
Then
% , for any fixed $a>0$,
$$J_3(t,a):=-\int_{[0,\,a]}W_{h(t)}(y){\rm d}v_t(y) \ \dod \ \int_{[0,\,a]}Z(y)e^{-y}{\rm d}y, \ \
t\to\infty.$$ Hence
$\underset{a\to\infty}{\lim}\,\lit\,J_3(t,a)=K$ in distribution.

Now we intend to show that, for any $c>0$,
\begin{equation}\label{22}
\underset{a\to\infty}{\lim}\underset{t\to\infty}{\lim\sup}\,\mmp\{|J_4(t,a)|>c\}=0.
\end{equation}
By Theorem 1.2 in \cite{Iks3}, for any $\delta>0$ there exists
$t_0>0$ such that $${\me \big|T(t)-{\tt m}^{-1}t\big|\over
g(t)}\leq \me |Z(1)|+\delta$$ whenever $t\geq t_0$. Hence, for $t$
such that $ah(t)\geq t_0$ and some $\varepsilon\in
(0,1-1/\alpha)$,
\begin{eqnarray}
\me |J_4(t,a)|&\leq& \int_{[ah(t),\,\infty)}{\me |T(y)-{\tt
m}^{-1}y|\over g(y)}{g(y)\over g(h(t))}{\rm
d}(-v_t(y/h(t)))\nonumber\\&\leq& \big(\me|Z(1)|+\delta\big)
\int_{[a,\,\infty)}{g(yh(t))\over g(h(t))}{\rm
d}(-v_t(y))\nonumber\\&\leq& \big(\me |Z(1)|+\delta\big)\,{\rm
const}\, \me
\eta_t^{1/\alpha+\varepsilon}1_{\{\eta_t>a\}}\label{24},
\end{eqnarray}
where in the third line the Potter's bound (Theorem 1.5.6 in
\cite{BGT}) has been utilized (recall that the regular variation
of $g$ was discussed at the beginning of the proof of Theorem
\ref{main}), and $\eta_t$ is a random variable with
$\mmp\{\eta_t>y\}=v_t(y)$. By Corollary 3.10.5 in \cite{BGT}, the
auxiliary function $h$ is unique up to the asymptotic equivalence
and can be taken $h(t)=\int_{[0,\,t]}\varphi(y){\rm
d}y/\varphi(t)$. With such $h$ we have
\begin{equation}\label{23}
\me \eta_t=\int_{[0,\infty)}v_t(y){\rm d}y={1\over
h(t)}{\int_{[0,t]}\varphi(y){\rm d}y\over \varphi(t)}+{1\over
h(t)\varphi(t)}\int_{[0,1]}{\Phi(y)\over y}{\rm d}y \ \to \ 1, \ \
t\to\infty.
\end{equation}
Note that the integral in the second term is finite in view of
$\Phi^\prime(0)<\infty$ (the latter finiteness was discussed in
Step 1 of the proof of Theorem \ref{main}).

Now \eqref{23} implies that the family
$(\eta_t^{1/\alpha+\varepsilon})_{t\geq 0}$ is uniformly
integrable, and \eqref{22} follows from \eqref{24} and Markov's
inequality. From this we conclude that the left-hand side of
\eqref{105} converges in distribution to $K$. This together with
\eqref{inter} and \eqref{234} proves \eqref{1033}.
\end{proof}

\subsection{Convergence of $K_{n,1}$}\label{small}

We shall prove next  convergence in distribution for the number of singleton blocks
$K_{n,1}$. Two cases, when Condition A
and Condition B holds, respectively, are treated in Theorem \ref{main4}
and Theorem \ref{main5}.

\begin{thm}\label{main4}
Assume that the function $t\mapsto t\Phi^\prime(t)$ is
nondecreasing and that Condition A holds with $\beta\geq 1$.

\noindent Under the assumptions of part (a) of Theorem \ref{main}
we have\footnote{Suppose $\beta>1$. According to Remark
\ref{1002}, $(\beta-1)\int_{[0,\,1]}Z(1-y)y^{\beta-2}{\rm d}y\od
(\alpha(\beta-1)+1)^{-1/\alpha}Z(1)$, where the case $\alpha=2$
corresponds to parts (a) and (b) of Theorem \ref{main4}.}: if
$\beta>1$ then
$${K_{n,1}-{\tt m}^{-1}\Phi(n)\over n\Phi^\prime(n)\sqrt{{\tt
s}^2{\tt m}^{-3}\log n}} \ \dod \ (\beta-1)
\int_{[0,1]}Z(1-y)y^{\beta-2}{\rm d}y, \ \ n\to\infty,$$ where
$(Z(y))_{y\in[0,1]}$ is the Brownian motion, and if $\beta=1$ and
$\lin L_1(n)=\infty$ then the limiting random variable is $Z(1)$.

\noindent
Under the assumptions of part (b) of Theorem \ref{main} we have:
if $\beta>1$ then
$${K_{n,1}-{\tt m}^{-1}\Phi(n)\over
{\tt m}^{-3/2} n\Phi^\prime(n)c(\log n)} \ \dod \ (\beta-1)
\int_{[0,1]}Z(1-y)y^{\beta-2}{\rm d}y, \ \ n\to\infty,$$ where
$(Z(y))_{y\in[0,1]}$ is the Brownian motion, and if $\beta=1$ then the
limiting random variable is $Z(1)$.

\noindent Under the assumptions of part (c) of Theorem \ref{main} we have:
if $\beta>1$ then
$${K_{n,1}-{\tt
m}^{-1}\Phi(n)\over {\tt m}^{-1-1/\alpha} n\Phi^\prime(n)c(\log
n)} \ \dod \ (\beta-1) \int_{[0,1]}Z(1-y)y^{\beta-2}{\rm d}y, \ \
n\to\infty,$$ where $(Z(y))_{y\in[0,1]}$ is the $\alpha$-stable
L\'{e}vy process such that $Z(1)$ has characteristic function
\eqref{st1}, and if $\beta=1$ then the limiting random variable is
$Z(1)$.
\end{thm}
\begin{rem}
Theorem \ref{main4} does not cover one interesting case when
${\tt s}^2<\infty$ and $\Phi(x) \sim c\log x$, as $x\to\infty$, where
$c>0$ is a constant. We conjecture that
\begin{equation*}\label{888}
{K_{n,1}-{\tt m}^{-1}\Phi(n)\over c\log^{1/2}n} \ \dod \ ({\tt
s}^2{\tt m}^{-3})^{1/2}V_1+({\tt m}c)^{-1/2}V_2, \ \ n\to\infty,
\end{equation*}
where $V_1$ and $V_2$ are independent random variables with the
standard normal distribution. In combination with the proof of
Theorem \ref{main4} this would follow once we could show that
$${K(t,1)-A^{(1)}(t)\over c\log^{1/2}t} \ \dod \ ({\tt m}\,c)^{-1/2}V_2, \ \
t\to\infty.$$ However, we have not been able to work it out.
\end{rem}
\begin{proof}
Using \eqref{co1} we have
$$A^{(1)}(e^t)=\int_{[0,\,\infty)} \varphi^\prime(t-y){\rm d}T(y)=\int_{[0,\,t]}+\int_{[t,\,\infty)}=:A^{(1)}_1(t)+A^{(1)}_2(t).$$
The function $\varphi^\prime$ is nonnegative and integrable on
$(-\infty,0]$, and the function $e^{-y}\varphi^\prime(y)$ is
nonincreasing on $\mr$. This implies that $\varphi^\prime$ is
directly Riemann integrable on $(-\infty,0]$ (see, for instance,
the proof of Corollary 2.17 in \cite{DurLig}). Therefore, by the
key renewal theorem, as $t\to\infty$,
\begin{equation}\label{key}
\me A_2^{(1)}(t)\to {\tt m}^{-1}\int_{(-\infty,\,
0]}\varphi^\prime(y){\rm d}y=\Phi(1)/{\tt m}<\infty.
\end{equation}
Now convergence in distribution of $A^{(1)}(t)$ with the same
centering and normalization as asserted for $K_{n,1}$ (and
$n$ replaced by the continuous variable $t$) follows
along the same lines as in the proof of Theorem \ref{main} for
$A(t)$.

Arguing in the same way as in the proof of Lemma 2.6 in
\cite{GBarbour} we conclude that $$\me (K(t,1)-A^{(1)}(t))^2=\me
A^{(1)}(t).$$ Hence, according to \eqref{key} and Proposition
\ref{s},
$$\me (K(t,1)-A^{(1)}(t))^2 \ \sim \ {\tt m}^{-1}\Phi(t), \ \
t\to\infty.$$
The function $\varphi'$ is nondecreasing since $t\Phi'(t)$ was assumed such,
%$u\mapsto\varphi^\prime(u)$
hence by the
monotone density theorem (Theorem 1.7.2 in \cite{BGT}) we conclude that, as
$t\to\infty$,
$${\varphi(t)\over (\varphi^\prime(t))^2 t} \ \sim \ {t^\beta L_1(t)\over \beta^2 t^{2\beta-2}L_1^2(t)t}={1\over \beta^2}{1\over t^{\beta-1}L_1(t)}.$$
This converges to zero whenever $\beta>1$ or $\beta=1$ and $\lit
L_1(t)=\infty$. Therefore, by Chebyshev's inequality
$${K(t,1)-A^{(1)}(t)\over t\Phi^\prime(t)\sqrt{\log t}} \
\tp  0, \ \
t\to\infty.$$
Since the normalization $t\Phi^\prime(t)\sqrt{\log
t}$ exhibits the slowest growth among the three normalizations
arising in the theorem (see the beginning of the proof of Theorem
\ref{main} for explanation), under the current assumption we
conclude that convergence in distribution as stated in the theorem
holds with $K_{n,1}$ replaced by $K(t,1)$ and the normalizing sequences
replaced by the normalizing functions.

Now we shall discuss the remaining case ${\tt s}^2=\infty$,
$\beta=1$ and $\lit L_1(t)=c\in (0,\infty)$ (note that in view of
the monotonicity assumption on $\varphi^\prime$ and the relation
$\varphi^\prime(t)\sim L_1(t)$, the limit of $L_1$ must exist).
The normalization $q_n$, say, claimed for $K_{n,1}$ grows not
slower than $n\Phi^\prime(n)\log^{1/2}n L_2(n)$ for some $L_2$
slowly varying at $\infty$ with $\lin L_2(n)=\infty$. Then,
Chebyshev's inequality implies $${K(n,1)-A^{(1)}(n)\over q_n} \
\tp 0, \ \ n\to\infty.$$ This proves that the asserted convergence in
distribution holds with $K_{n,1}$ replaced by
$K(t,1)$ in this case too.

It remains to depoissonize. Let
%$\sum_k \epsilon_{(t_k, x_k)}$
$(t_k,x_k)$ be the atoms of a planar Poisson point process in the
positive quadrant with the intensity measure given by $e^{-x} {\rm
d}t\,{\rm d}x$. The process $(X_t)_{t\geq 0}$ with
$X_t:=\sum_{t_k\leq t}x_k$ is a compound Poisson process with unit
intensity and jumps having the standard exponential distribution.
Now, with $z\geq 0$ fixed, $\pi_z$ can be identified with the
number of jumps of $(X_t)$ occurring before time $z$, which
implies that $(\pi_z)_{z\geq 0}$ is a homogeneous Poisson process
with unit intensity. Denote by $(T_n)_{n\in\mn}$ its arrival
times. We already know that
\begin{equation}\label{8989}
{K(t,1)-r(t)\over d(t)} \ \dod \ X, \ \ t\to\infty
\end{equation}
for $r(t):={\tt m}^{-1}\Phi(t)$, the case-dependent normalizing
function $d(t)$ and the case-dependent random variable $X$. Since
$K(T_n,1)=K_{n,1}$ it suffices to check that
$${K(T_n,1)-r(n)\over d(n)} \ \dod \ X, \ \ n\to\infty.$$

In the subsequent computations we will use arbitrary but fixed
$x\in\mr$. Given such  $x$ we will choose $n_0\in\mn$ such that
the sequence $(n+x\sqrt{n})_{n\geq n_0}$ is nondecreasing and
every its element is not smaller than one, and the sequence
$(n-x\sqrt{n})_{n\geq n_0}$ is nonnegative. Also, we will choose
$t_0\in (0,\infty)$ such that $t\pm x\sqrt{t}\geq 0$ for $t\geq
t_0$. With this notation all the relations
that follow will be considered either for $t\geq t_0$
or $n\geq n_0$.

The function $d(t)$ is slowly varying, which implies that the
convergence $\lit {d(ty)\over d(t)}=1$ holds locally uniformly in
$y$. In particular,
\begin{equation}\label{121212}
\lit {d(t\pm x\sqrt{t})\over d(t)}=1.
\end{equation}
The function $r(t)$ has the following property
\begin{equation}\label{121211}
\lit {r(t\pm x\sqrt{t})-r(t)\over d(t)}=0.
\end{equation}
Indeed, the function $t\mapsto \Phi^\prime(t)$ is nonincreasing,
and using the mean value theorem we conclude that $$
{r(t+x\sqrt{t})-r(t)\over
 d(t)}\leq {{\tt m}^{-1}x\sqrt{t}\Phi^\prime(t)\over t\Phi^\prime(t)}o(1), \ \ {r(t)-r(t-x\sqrt{t})\over
 d(t)}\leq {{\tt m}^{-1}x\sqrt{t}\Phi^\prime(t-x\sqrt{t})\over
 t\Phi^\prime(t)}o(1).$$ By the monotone density theorem (Theorem 1.7.2 in \cite{BGT}), the function $\Phi^\prime(t)$ is regularly varying at $\infty$ with index $-1$.
Hence $\lit {\Phi^\prime(t-x\sqrt{t})\over \Phi^\prime(t)}=1$, and
the right-hand side of the last inequality tends to zero, as
$t\to\infty$.

Now \eqref{121212} and \eqref{121211} ensure that \eqref{8989} is
equivalent to
\begin{equation}\label{898989}
{K(t\pm x\sqrt{t},1)-r(t)\over d(t)} \ \dod \ X, \ \ t\to\infty.
\end{equation}
We will need the following observation
\begin{equation}\label{us}
{K(t+x\sqrt{t})-K(t-x\sqrt{t})\over d(t)}\ \tp \ 0, \ \
t\to\infty,
\end{equation}
which can be proved as follows. Since $K(t)$ is nondecreasing it
suffices to show that the expectation of the left-hand side
converges to zero. To this end, write
\begin{eqnarray*}
&& \me \bigg(K(t+x\sqrt{t})-K(t-x\sqrt{t})\bigg)\\&=&\me
\int_{[0,\infty)}\bigg(\varphi\big(\log(t+x\sqrt{t})-y\big)-\varphi\big(\log(t-x\sqrt{t})-y\big)\bigg){\rm
d}T(y)\\&=& \me \int_{[0,\,
\log(t+x\sqrt{t})]}\bigg(\varphi\big(\log(t+x\sqrt{t})-y\big)-\varphi\big(\log(t-x\sqrt{t})-y\big)\bigg){\rm
d}T(y)+O(1)\\&\leq& \log\left({t+x\sqrt{t}\over t-x\sqrt{t}}\right) \me
\int_{[0,\,\log(t+x\sqrt{t})]}\varphi^\prime\big(\log(t+x\sqrt{t})-y\big){\rm
d}T(y)+O(1)\\&\sim& {2x\over \sqrt{t}}\, \me
\int_{[0,\,\log(t+x\sqrt{t})]}\varphi^\prime\big(\log(t+x\sqrt{t})-y\big){\rm
d}T(y) \\&\sim&
%{1\over {\tt m}}
{2x\over
{\tt m}\sqrt{t}}\,\int_{[0,\,\log(t+x\sqrt{t})]}\varphi^\prime(y){\rm
d}y\ \sim \
%{1\over {\tt m}}
{2x\over
{\tt m}\sqrt{t}}\Phi(t+x\sqrt{t})\sim
%{1\over {\tt m}}
{2x\over
{\tt m}\sqrt{t}}\Phi(t) , \ \ t\to\infty.
\end{eqnarray*}
Here the third line is a consequence of the key renewal theorem
(see the paragraph preceding formula \eqref{key} for more
details). The fourth line follows from the mean value theorem and
the monotonicity of $\varphi^\prime$. While Proposition \ref{s}
justifies the first asymptotic equivalence in the sixth line of the last display, the
last equivalence in that line is implied by the slow variation of
$\Phi$ (see the sentence preceding \eqref{121212} for the
explanation). Now \eqref{us} follows from the last asymptotic
relation and the observation $\lit {t\Phi^\prime(t)\over
\Phi(t)}\sqrt{t}=\infty$, the latter being trivial as the first
factor is slowly varying.

Set $D_n(x):=\{|T_n-n|>x\sqrt{n}\}$. Since $K(t)$ and
$L(t):=K(t)-K(t,1)$ are nondecreasing, we have, for any
$\varepsilon>0$,
\begin{eqnarray*}
\mmp\bigg\{{K(T_n,1)-K(n-x\sqrt{n},1)\over
d(n)}>2\varepsilon\bigg\}&=&\mmp\bigg\{{K(T_n)-L(T_n)-K(n-x\sqrt{n},1)\over
d(n)}>2\varepsilon\bigg\}\nonumber\\&=& \mmp\big\{\ldots
1_{D^c_n(x)}+\ldots 1_{D_n(x)}>2\varepsilon\big\}\nonumber\\&\leq&
\mmp\bigg\{{K(n+x\sqrt{n})-L(n-x\sqrt{n})-K(n-x\sqrt{n},1)\over
d(n)}>\varepsilon\bigg\}\nonumber\\&+&\mmp\big\{\ldots
1_{D_n(x)}>\varepsilon\big\}\nonumber\\&\leq&\mmp\bigg\{{K(n+x\sqrt{n})-K(n-x\sqrt{n})\over
d(n)}>\varepsilon\bigg\}+\mmp\big(D_n(x)\big).
\end{eqnarray*}
Hence, by \eqref{us} and the central limit theorem
\begin{equation}\label{op2}
\underset{n\to\infty}{\lim\sup}\,\mmp\bigg\{{K(T_n,1)-K(n-x\sqrt{n},1)\over
a(n)}>2\varepsilon\bigg\}\leq
\mmp\big\{|\mathcal{N}(0,1)|>x\big\},
\end{equation}
where $\mathcal{N}(0,1)$ denotes a random variable with the
standard normal distribution. Since the law of $X$ is continuous,
we conclude that, for any $y\in\mr$ and any $\varepsilon>0$,
\begin{eqnarray*}
\underset{n\to\infty}{\lim\sup}\,\mmp\bigg\{{K(T_n,1)-r(n)\over
d(n)}>y\bigg\}&\leq&
\underset{n\to\infty}{\lim\sup}\,\mmp\bigg\{{K(T_n,1)-K(n-x\sqrt{n},1)\over
d(n)}>2\varepsilon\bigg\}\\&+&\lin\,\mmp\bigg\{{K(n-x\sqrt{n},1)-r(n)\over
d(n)}>y-2\varepsilon\bigg\}\\&\overset{\eqref{898989},\eqref{op2}}{\leq}&
\mmp\big\{|\mathcal{N}(0,1)|>x\big\}+\mmp\big\{X>y-2\varepsilon\big\}.
\end{eqnarray*}
Letting now $x\to\infty$ and then $\varepsilon\downarrow 0$ gives
$$\underset{n\to\infty}{\lim\sup}\,\mmp\bigg\{{K(T_n,1)-r(n)\over
d(n)}>y\bigg\}\leq \mmp\big\{X>y\big\}.$$ Arguing similarly we
infer
\begin{equation}\label{op3}
\underset{n\to\infty}{\lim\sup}\,\mmp\bigg\{{K(n+x\sqrt{n},1)-K(T_n,1)\over
d(n)}>2\varepsilon\bigg\}\leq \mmp\big\{|\mathcal{N}(0,1)|>x\big\}
\end{equation}
and then
\begin{eqnarray*}
\underset{n\to\infty}{\lim\inf}\,\mmp\bigg\{{K(T_n,1)-r(n)\over
d(n)}>y\bigg\}&\geq& \lin\,\mmp\bigg\{{K(n+x\sqrt{n},1)-r(n)\over
d(n)}>y+2\varepsilon\bigg\}\\&-&
\underset{n\to\infty}{\lim\sup}\,\mmp\bigg\{{K(n+x\sqrt{n},1)-K(T_n,1)\over
d(n)}>2\varepsilon\bigg\}\\&\overset{\eqref{898989},\eqref{op3}}{\geq}&
\mmp\big\{X>y+2\varepsilon\big\}-\mmp\big\{|\mathcal{N}(0,1)|>x\big\}.
\end{eqnarray*}
Letting $x\to\infty$ and then $\varepsilon\downarrow 0$ we arrive
at
$$\underset{n\to\infty}{\lim\inf}\,\mmp\bigg\{{K(T_n,1)-r(n)\over
d(n)}>y\bigg\}\geq \mmp\big\{X>y\big\}.$$ The proof is complete.
\end{proof}

\begin{thm}\label{main5}
Assume that the function $t\mapsto t\Phi^\prime(t)$ is
nondecreasing and that Condition B holds.

\noindent Under the assumptions of part (a) of Theorem
\ref{main3}\footnote{See Remark \ref{1001} for the identification
of the laws of $\int_{[0,\,\infty)}Z(y)e^{-y}{\rm d}y$.}
$${K_{n,1}-{\tt m}^{-1}\Phi(n)\over n\Phi^\prime(n)\sqrt{{\tt
s}^2{\tt m}^{-3}h(\log n)}} \ \dod \
\int_{[0,\infty)}Z(y)e^{-y}{\rm d}y, \ \ n\to\infty,$$ where
$(Z(y))_{y\geq 0}$ is the Brownian motion.

\noindent Under the assumptions of part (b) of Theorem \ref{main3}
$${K_{n,1}-{\tt
m}^{-1}\Phi(n)\over {\tt m}^{-3/2} n\Phi^\prime(n)c(h(\log n))} \
\dod \ \int_{[0,\infty)}Z(y)e^{-y}{\rm d}y, \ \ n\to\infty,
$$ where $(Z(y))_{y\geq 0}$ is the Brownian motion.

\noindent Under the assumptions of part (c) of Theorem \ref{main3}
$${K_{n,1}-{\tt
m}^{-1}\Phi(n)\over {\tt m}^{-1-1/\alpha} n\Phi^\prime(n)c(h(\log
n))} \ \dod \ \int_{[0,\infty)}Z(y)e^{-y}{\rm d}y, \ \
n\to\infty,$$ where $(Z(y))_{y\geq 0}$ is the $\alpha$-stable
L\'{e}vy process such that $Z(1)$ has characteristic function
\eqref{st1}.
\end{thm}
\begin{proof}
By Theorem 3.10.11 in \cite{BGT} $\varphi^\prime$ belongs to de
Haan's class $\Gamma$. For later use, note that this implies
\begin{equation}\label{in}
\lit \varphi^\prime(t)=\infty.
\end{equation}
By Corollary 3.10.7 in \cite{BGT} one can take $h$ as the
auxiliary function of $\varphi^\prime$. With this at hand the
proof of convergence in distribution of $A^{(1)}(t)$ with the same
centering and normalization as claimed for $K_{n,1}$ (but with
discrete argument $n$ replaced by continuous argument $t$)
literally repeats the proof of Theorem \ref{main3}, thus omitted.

The next step is to prove that $${K(t,1)-A^{(1)}(t)\over d(t)} \
\tp 0, \ \ t\to\infty,$$ where, depending on the context,
$d(t)$
equals either
$${\rm const}\, t\Phi^\prime(t)\sqrt{h(\log t)}  {\rm~~~or~~~}
{\rm const}\, t\Phi^\prime(t)c(h(\log t)).$$
Since
the function $t\Phi^\prime(t)\sqrt{h(\log t)}$ grows slower than
the other one it suffices to prove that
\begin{equation}\label{2121}
{K(t,1)-A^{(1)}(t)\over t\Phi^\prime(t)\sqrt{h(\log t)}} \ \tp 0,
\ \ t\to\infty.
\end{equation}
From the proof of Theorem \ref{main4} we know that
$$\me (K(t,1)-A^{(1)}(t))^2 \ \sim \ {\tt m}^{-1}\Phi(t), \ \
t\to\infty.$$  By Corollary 3.10.5 in \cite{BGT}, $$h(t) \sim
{\varphi(t)\over \varphi^\prime(t)}, \ \ t\to\infty.$$ Therefore,
using \eqref{in} at the last step,
$${\varphi(t)\over (\varphi^\prime(t))^2 h(t)} \ \sim \ {1\over
\varphi^\prime(t)} \ \to \ 0, \ \ t\to\infty,$$ and relation
\eqref{2121} follows by Chebyshev's inequality.

By Lemma \ref{5050}, the functions $d(t)$ are slowly varying at
$\infty$. Keeping this in mind, the depoissonization step runs
exactly the same route as in the proof of Theorem \ref{main4}.
\end{proof}

\section{The compound Poisson case}\label{com}

In this section we assume that $S$ is a compound Poisson process
whose L\'{e}vy measure $\nu$ is a  probability measure. This does
not reduce generality, since the range $\cal R$ is not affected by
the normalization of $\nu$. Let $-\log W_1,~ -\log W_2, \ldots$
(where $0<W_j<1$ a.s.) be the sizes of the consecutive jumps of
$S$, which are independent random variables with distribution
$\nu$. Define a zero-delayed random walk $(R_k)_{k\geq 0}$ with
such increments $-\log W_k$. In these terms, the Laplace exponent
of $S$ is $\widehat{\Phi}(t)=1-\me e^{-t(1-W_1)}$.

The argument exploited in Section \ref{gen} extends smoothly when
the variance of $S(1)$ is infinite. Otherwise the
problem arises that the terminal value  $A(n)$
of the compensator does not absorb enough of the variability of $K_n$.
%, i.e. we have some number of `boxes' and some number of
%`balls', and the time variable does not really matter.
The continuous-time compensator process
carries extra variability  coming from the exponential
waiting times
%$\tau_i$'s
between the jumps of $S$.
% and using it as the
%approximation to $K(z)$ is no longer adequate.
Without going into
details we only mention that the excessive variability is seen from the
asymptotics
\begin{equation}\label{badappr}
\me \big(K(t)-A(t)\big)^2 \ \sim \ {\tt m}^{-1}\log t, \ \
t\to\infty,
\end{equation}
where ${\tt m}=\me S(1)=\me |\log W_1|$.
%and the central limit theorem $${A(t)-{\tt
%m}^{-1}\int_{[0,\,t]}(1-f(e^y)){\rm d}y\over \sqrt{{\tt s}^2{\tt
%m}^{-3}\log t}}\ \dod \ \mathcal{N}(0,1), \ \ t\to\infty.$$

To circumvent the complication we note that
in the case of finite L\'{e}vy measure the setting is intrinsically
discrete-time, hence it is natural to
%A simple way to overcome the problem is to discard the effect the
%sequence $(\tau_i)$ brings in,
%or to put it the other way, to
replace $T(y)$ in \eqref{co} by
$$\rho(y):=\inf\{k\in\mn_0:
R_k>y\}$$ and to consider a
%Formally this is implemented by introducing a
{\it discrete-time compensator}. Denote by $C_k$ the event that
the interval $[R_{k-1}, R_k]$ is occupied by at least one point of
the Poisson process $(\pi_t(u))_{u\geq 0}$. Then $K(t)=\sum_{k\geq
1} 1_{C_k}$, and we define the discrete-time compensator by
$$B(t):=\sum_{k\geq
1}\mmp\{C_k|R_{k-1}\}=\sum_{k\geq
1}\widehat{\Phi}\big(te^{-R_{k-1}}\big)=\int_{[0,\infty)}\widehat{\Phi}\big(te^{-y}\big){\rm
d}\rho(y).$$ Indeed,
$\mmp\{C_k|R_{k-1},W_k\}=1-\exp\big(-te^{-R_{k-1}}(1-W_k)\big)$
which entails $\mmp\{C_k|R_{k-1}\}=\widehat{\Phi}(te^{-R_{k-1}})$,
thus justifying the second equality above. For the discrete-time
compensator we have
$$\me (K(t)-B(t))^2=\int_{[0,\infty)} \widehat{\Phi}(te^{-y})(1-\widehat{\Phi}(te^{-y})){\rm d}\me \rho(y)
=o(\log t), \ \ t\to\infty,$$ which compared with (\ref{badappr})
shows that $B(t)$ approximates $K(t)$ better than $A(t)$ . Furthermore, by the
key renewal theorem the integral converges to $${\tt
m}^{-1}\int_{[0,\infty)}\widehat{\Phi}(y)\big(1-\widehat{\Phi}(y)\big)y^{-1}{\rm
d}y$$ provided that
$\int_{[1,\infty)}\big(1-\widehat{\Phi}(y)\big)y^{-1}{\rm
d}y<\infty$, and by Proposition \ref{s} it is asymptotic to ${\tt
m}^{-1}\int_{[0,\,\log t]}\big(1-\widehat{\Phi}(e^y)){\rm d}y$
otherwise. These findings allow us to simplify
%by two steps
the proof of the following result obtained previously in \cite{GIM}.
\begin{thm}\label{main2}
{\rm (a)} If $\sigma^2={\rm Var}(\log W)<\infty$ then for
\begin{equation}\label{1022}
b_n={1\over {\tt m}}\int_1^n{\widehat{\Phi}(z)\over z}{\rm d}z \ \
\text{or} \ \ b_n={1\over {\tt m}}\int_0^{\log n}
\mmp\{|\log(1-W)|\leq z\}{\rm d}z,
\end{equation}
where ${\tt m}=\me S(1)=\me |\log W|$, and for
$$a_n=\sqrt{{\sigma^2\over {\tt m}^3}\log n},$$
the limiting distribution of ${K_n-b_n\over a_n}$ is standard
normal.
\newline {\rm (b)} If $\sigma^2=\infty$ and $$\int_0^x y^2 \nu({\rm d}y) \ \sim \ L(x), \ \ x\to\infty,$$ for some $L$ slowly varying at $\infty$, then,
with $b_n$ given in \eqref{1022} and
$$a_n={\tt m}^{-3/2}c(\log n),$$ where $c(x)$ is any positive
function satisfying $\lix\,xL(c(x))/c^2(x)=1$, the limiting
distribution of ${K_n-b_n\over a_n}$ is standard normal.\newline
{\rm (c)} If $\nu$ satisfies \eqref{domain1} then, with $b_n$ as
in \eqref{1022} and
\begin{equation}\label{106}
a_n={\tt m}^{-1-1/\alpha} c(\log n),
\end{equation}
where $c(x)$ is any positive function satisfying $\lix
\,xL(c(x))/c^\alpha(x)=1$, the limiting distribution of
$(K_n-b_n)/a_n$ is the $\alpha$-stable law with characteristic
function \eqref{st1}.
\end{thm}
\begin{proof}
Let $g$ and $Z$ be as defined at the beginning of the proof of
Theorem \ref{main}. We only give a proof of the poissonized version of the result,
 with $K_n$ replaced by $B(t)$. Recalling the notation
$\varphi(y)=\Phi(e^y)$ and noting that $\varphi$ is integrable in
the neighborhood of $-\infty$, an appeal to Theorem 4.1 in
\cite{Iks} gives $${B(e^t)-{\tt
m}^{-1}\int_{[1,\,e^t]}\big(\Phi(y)/y\big){\rm d}y\over
g(t)}={\int_{[0,\,\infty)}\varphi(t-y){\rm d}\rho(y)-{\tt
m}^{-1}\int_{[0,\,t]}\varphi(y){\rm d}y \over g(t)} \ \dod \ Z(1),
\ \ t\to\infty.$$

Lemma \ref{we} with $V=1-W_1$ ensures that the centering $${\tt
m}^{-1}\int_{[0,\,t]}\varphi(u){\rm d}u={\tt
m}^{-1}\int_{[0,\,t]}\big(1-\me \exp(-e^u(1-W_1))\big){\rm d}u$$
can be safely replaced by $${\tt
m}^{-1}\int_{[0,\,t]}\mmp\{|\log(1-W)|\leq u\}{\rm d}u={\tt
m}^{-1}\int_{[0,\,t]}\mmp\{1-W\geq e^{-u}\}{\rm d}u,$$ because the
absolute value of their difference is $O(1)$ and $\lit
g(t)=\infty$. Replacing $e^t$ by $t$ completes the proof.
\end{proof}

\section{Appendix}

The first auxiliary result concerns the laws of some Riemann
integrals of the L\'{e}vy processes.
\begin{lemma}\label{ri}
Let $q$ be a Riemann integrable function on $[0,1]$ and
$(Z(y))_{y\in [0,1]}$ a L\'{e}vy process with $g(t):=\log
\me\exp({\rm i}tZ(1))$. Then
\begin{equation}\label{1}
\me\exp \bigg({\rm i}t\int_{[0,1]} q(y)Z(y){\rm
d}y\bigg)=\exp\bigg(\int_{[0,1]} g\bigg(t\int_{[y,1]} q(z){\rm
d}z\bigg){\rm d}y\bigg), \ \ t\in\mr.
\end{equation}
Similarly, for $q$ a directly Riemann integrable function on
$[0,\infty)$ and $(Z(y))_{y\geq 0}$ a L\'{e}vy process it holds that
\begin{equation*}
\me\exp \bigg({\rm i}t\int_{[0,\infty)} q(y)Z(y){\rm
d}y\bigg)=\exp\bigg(\int_{[0,\infty)} g\bigg(t\int_{[y,\infty)}
q(z){\rm d}z\bigg){\rm d}y\bigg), \ \ t\in\mr.
\end{equation*}
\end{lemma}
\begin{proof}
We only prove the first assertion. The integral in the left-hand
side of \eqref{1} exists as a Riemann integral and as such can be
approximated by Riemann sums
\begin{eqnarray*}
n^{-1}\sum_{k=1}^n q(k/n)Z(k/n)&=&\sum_{k=1}^n
\bigg(Z(k/n)-Z((k-1)/n)\bigg)\bigg(n^{-1}\sum_{j=k}^n
q(j/n)\bigg)\\&=:&\sum_{k=1}^n
\bigg(Z(k/n)-Z((k-1)/n)\bigg)a_{k,n}=:I_n
\end{eqnarray*}
Since $Z$ has independent and stationary increments, we conclude
that
$$\log \me \exp({\rm i}t I_n)=n^{-1}\sum_{k=1}^ng(ta_{k,n}).$$
Letting $n\to\infty$ we arrive at \eqref{1}, by L{\'e}vy's continuity
theorem for characteristic functions.
\end{proof}

Lemma \ref{5050} collects some useful properties of the functions
$\Phi$ satisfying Condition B.
\begin{lemma}\label{5050}
Suppose Condition B holds. Then the functions $\Phi(t)$ and
$h(\log t)$ are slowly varying at $\infty$. The function $t\mapsto
t\Phi^\prime(t)$ is slowly varying at $\infty$, whenever it is
nondecreasing.
\end{lemma}
\begin{proof}
By Proposition 3.10.6 and Theorem 2.11.3 in \cite{BGT}, the
function $h(\log t)$ is slowly varying. As was already mentioned
in the proof of Theorem \ref{main3}, without loss of generality
the auxiliary function $h$ can be taken
$h(t)=\int_{[0,t]}\varphi(y){\rm d}y/\varphi(t)$. By the
representation theorem for slowly varying functions (Theorem 1.3.1
in \cite{BGT}), the function $t\mapsto \int_{[0,\log
t]}\varphi(y){\rm d}y$ is slowly varying. Hence
$\Phi(t)=\varphi(\log t)$ is slowly varying as well.

By Theorem 3.10.11 and Corollary 3.10.7 in \cite{BGT},
$\varphi^\prime$ belongs to de Haan's class $\Gamma$ with the
auxiliary function $h_1$ such that $h_1(t)\sim h(t)$,
$t\to\infty$. By Corollary 3.10.5 in \cite{BGT},
$t\Phi^\prime(t)\sim \Phi(t)/h(\log t)$, $t\to\infty$. Since both
numerator and denominator are slowly varying functions, the
function $t\mapsto t\Phi^\prime(t)$ is slowly varying.
\end{proof}

The following lemma was a basic ingredient in the proof of
our main results (Theorem \ref{main} and the like).
\begin{lemma}\label{impo}
Assume that $X_{t}(\cdot) \Rightarrow X(\cdot)$, as
$t\to\infty$, in $D[0,\infty)$ in the Skorohod $M_1$ or $J_1$
topology. Assume also that, as $t\to\infty$, $Y_t \dod Y$, where
$(Y_t)$ is a family of nonnegative random variables such that
$\mmp\{Y_t=0\}$ may be positive, and $Y$ has an absolutely
continuous distribution. Then, for  $a>0$,
$$\int_{[0,\,a]}X_{t}(u)\mmp\{Y_t\in {\rm d}u\} \ \dod \
\int_{[0,\,a]}X(u)\mmp\{Y\in {\rm d}u\}, \ \ t\to\infty.$$
\end{lemma}
\begin{proof}
It suffices to prove that
\begin{equation}\label{25}
\lit \me h_t(Y_t)=\me h(Y)
\end{equation}
whenever
%$\lit h_t(x)=h(x)$
$h_t\to h$
in $D[0,\infty)$ in the $M_1$ or $J_1$
topology, for the desired result then follows by the continuous
mapping theorem.

\noindent Since
%$x\in D[a,b]$,
$h$ restricted to $[c,d]$ is in $D[c,d]$ the set ${\cal D}_h$ of
its discontinuities is at most countable. By Lemma 12.5.1 in
\cite{Whitt2}, convergence in the $M_1$ topology (hence in the
$J_1$ topology) implies the local uniform convergence at all
continuity points of the limit distribution. Hence
$$E:=\{y: \text{there exists} \ y_t \ \text{such that}
\ \lit y_t= y, \text{but} \ \lin h_t(y_t)\neq h(y)   \}\subseteq
{\cal D}_h,$$
and we conclude that $\mmp\{Y\in E\}=0$. Now \eqref{25}
follows by Theorem 5.5 in \cite{Bill}.
\end{proof}

%Recall the notation $\varphi(y)=\Phi(e^y)$. In the next lemma the
%standing assumption that $\me S_1<\infty$ can be omitted.
%\begin{lemma}\label{conc}
%The function $\varphi^\prime$ is log-concave, hence eventually
%monotone.
%\end{lemma}
%\begin{proof}
%Since $\log \varphi^\prime(y)=y+\log \Phi^\prime(e^y)$, it
%suffices to prove that the second term is log-concave. The
%function $\Phi^\prime(y)$ is the Laplace-Stieltjes transform of a
%measure. Hence it is nonincreasing and convexSince
%\end{proof}

%In what follows the record $q\in \Pi_m$ means that the function
%$q$ belongs to de Haan's class $\Pi$ with auxiliary function $m$.
%
%Recall the notation $f(y)=\Phi(e^y)$.
%\begin{lemma}
%Condition A implies that $f\in \Pi_L$.
%\end{lemma}
%\begin{proof}
%We first prove that if $p\in \Pi_\ell$, i.e.,
%\begin{equation}\label{131}
%\liy {p(yt)-p(y)\over \ell(y)}=\log t \ \ \text{for all} \ \ t>0,
%\end{equation}
%then $p^\ast\in \Pi_{\ell^\ast}$, where $p^\ast(y):=p(\log y)$ and
%$\ell^\ast(y):=\ell(\log y)/\log y$.
%
%The convergence in \eqref{131} is locally uniform on $(0,\infty)$.
%Hence, for any $\varepsilon>0$, any fixed $t>0$ and large enough
%$y>0$ we have $$\bigg|{p^\ast(yt)-p^\ast(t)\over \ell(\log
%t)}-\log \bigg(1+{\log t\over \log y}\bigg)\bigg|=\bigg|{p(\log
%y+\log t)-p(\log t)\over \ell(\log t)}-\log \bigg(1+{\log t\over
%\log y}\bigg)\bigg|\leq \varepsilon.$$
%\end{proof}

Proposition \ref{s} is a slight extension of Theorem 4 in
\cite{Sgib} and Lemma 5.4 in \cite{Iks2}.
\begin{assertion}\label{s}
Let $v$ be a nonnegative function with $\lit
\int_{[0,\,t]}v(z){\rm d}z=\infty$. Assume further that $v$ is
either nondecreasing with $$\lit {v(t)\over \int_{[0,\,t]}v(z){\rm
d}z}=0,$$ or nonincreasing. If~ ${\tt m}=\me S(1)<\infty$ then
\begin{equation}\label{4}
\int_{[0,\,t]}v(t-z){\rm d}\me T(z) \ \sim \ {\tt
m}^{-1}\int_{[0,\,t]}v(z){\rm d}z, ~~~t\to\infty,
\end{equation}
provided the subordinator $S$ is nonarithmetic.
The asymptotic relation holds with
additional factor $\delta$ for $S$
 arithmetic subordinator with span $\delta>0$.
\end{assertion}

Sgibnev \cite{Sgib} and Iksanov \cite{Iks2} assumed that $T(u)$ is
the first passage time through the level $u$ by a random walk with
nonnegative steps. The transition to the present setting is easy
in the view of $\me T(u)=\me N^\ast(u)+\delta_0(u)$, where
$N^\ast(u)$ is the first passage time through the level $u$ by a
zero-delayed random walk with the generic increment $\xi$ having
the distribution
$$\mmp\{\xi\in {\rm d}x\}=\int_{[0,\,\infty)}\mmp\{S(t)\in
{\rm d}x\}e^{-t}{\rm d}t.$$ It is clear that if the law of $S(1)$
is arithmetic with span $\delta>0$ (respectively, nonarithmetic) then the same
is true for the law of $\xi$.

The next lemma was used in the proof of Theorem \ref{main2}.
\begin{lemma}\label{we}
For $x>0$ and a random variable $V\in (0,1)$,
$$-\int_{[0,\,1]}{1-e^{-y}\over y}{\rm d}y\leq f_1(x)-f_2(x)\leq
\int_{[1,\,\infty)}{e^{-y}\over y}{\rm d}y,$$ where
$$f_1(x):=\int_{[0,\,x]}\me \exp(-e^yV){\rm d}y \ \ \text{and} \ \ f_2(x):=\int_{[0,\,x]}\mmp\{V<e^{-y}\}{\rm d}y.$$
\end{lemma}
\begin{proof}
For fixed $z>0$ define $r(x)=x\wedge z$, $x\in\mr$. This function
is subadditive on $[0,\infty)$ and nondecreasing. Hence, for
$x\geq 0$ and $y\in\mr$ we have
$$r((x+y)^+)\leq r(x+y^+)\leq r(x)+r(y^+)\leq r(x)+y^+$$ and
\begin{eqnarray*}
r((x+y)^+)-r(x)&\geq& r(x-y^-)-r(x)\\&=&(r(x-y^-)-r(x)) 1_{\{x\leq
z\}}+(r(x-y^-)-r(x))1_{\{x>z,\, x-y^-\leq z\}}\\&=&-y^-1_{\{x\leq
z\}}+(x-y^--z)1_{\{x>z,\, x-y^-\leq z\}}\\&\geq&-y^-1_{\{x\leq
z\}}-y^-1_{\{x>z,\, x-y^-\leq z\}}\\&\geq&-y^-.
\end{eqnarray*}
Thus we have proved that, for $x\geq 0$ and $y\in\mr$
\begin{equation}\label{107}
-y^-\leq r((x+y)^+)-r(x)\leq y^+.
\end{equation}

\noindent Since $f_2(z)=\me (|\log V|\wedge z)$ and
$$f_1(z)=\int_{[0,\,z]}\mmp\{E/V>e^y\}{\rm d}y=\int_{[0,\,z]}\mmp\{|\log
V|+\log E>y\}{\rm d}y=\me ((|\log V|+\log E)^+\wedge z),$$where
$E$ is a random variable with the standard exponential
distribution which is independent of $V$, \eqref{107} entails
$$-\me \log^- E\leq f_1(z)-f_2(z)\leq \me \log^+ E.$$ The proof
is complete.
\end{proof}

\vskip0.5cm \noindent {\bf Acknowledgement}. The authors are
indebted to the referees for helpful comments.

\end{document}